\documentclass[11pt,A4paper]{article}

\date{}
\usepackage{graphicx}
\usepackage{picins}
\usepackage{amsmath}
\usepackage{amsthm}
\usepackage{amssymb}
\usepackage{amstext}
\usepackage{epic}
\usepackage{eepic}
\usepackage[center]{caption2}
\usepackage{subfigure}
\usepackage{setspace}
\usepackage[noblocks]{authblk}
\allowdisplaybreaks[4]
\DeclareMathOperator{\tr}{tr}
\makeatletter
\newtheorem{theorem}[subsection]{Theorem}
\newtheorem{proposition}[subsection]{Proposition}

\newtheorem{lemma}[subsection]{Lemma}
\theoremstyle{definition}
\newtheorem{definition}[subsection]{Definition}
\newtheorem{remark}[subsection]{Remark}
\makeatletter

\begin{document}
\setlength{\baselineskip}{18pt}

\title{\textbf{A new approach to a network of congruences on an inverse semigroup}}
\author[1]{Ying-Ying Feng}
\author[2]{Li-Min Wang\thanks{Correspondence author. Email: wanglm@scnu.edu.cn}}
\author[2]{Lu Zhang}
\author[2]{Hai-Yuan Huang}
\affil[1]{Department of Mathematics, Foshan University, \authorcr Foshan 528000, P. R. China}
\affil[2]{School of Mathematics, South China Normal University, \authorcr Guangzhou 510631, P. R. China}
\maketitle

\begin{abstract}
 This paper enriches the list of known properties of congruence sequences starting from the universal relation and successively performing
 the operators lower $k$ and lower $t$. Two series of inverse semigroups, namely $\ker{\alpha_n}$-is-Clifford semigroups and
 $\beta_n$-is-over-$E$-unitary semigroups, are investigated. Two congruences, namely $\alpha_{n+2}$ and $\beta_{n+2}$, are found to be the
 least $\ker{\alpha_n}$-is-Clifford and least $\beta_n$-is-over-$E$-unitary congruences on $S$, respectively. A new system of implications is
 established for the quasivarieties of inverse semigroups induced by the min network.

 \textbf{Keywords:} inverse semigroup, congruence, $\ker{\alpha_n}$-is-Clifford semigroup, $\beta_n$-is-over-$E$-unitary semigroup, min
 network.

 \textbf{2000 MR Subject Classification:} 20M18
\end{abstract}

In semigroup theory it is not possible to avoid the explicit study of congruences. Congruences play a central role in many of the structure
theorems and other important considerations in the theory of inverse semigroups. An efficient handling of congruences on inverse semigroups is
the kernel - trace approach. From the kernel - trace decomposition of congruences, we obtain two operators, lower $k$ and lower $t$, on the
congruence lattice $\mathcal{C}(S)$ of an inverse semigroup. We denote by $\rho_k$ the least congruence on $S$ having the same kernel as
$\rho$, and by $\rho_t$ the least congruence having the same trace as $\rho$. Starting with the universal congruence $\omega$ on $S$, we form
two sequences: $$\omega,~\omega_k,~(\omega_k)_t, \cdots \quad \text{and} \quad \omega,~\omega_t,~(\omega_t)_k, \cdots.$$ These congruences,
together with the intersections $\omega_t \cap \omega_k$, $(\omega_t)_k \cap (\omega_k)_t$, $\cdots$, form a sublattice of the lattice of all
congruences on $S$. Petrich -- Reilly \cite{network} first investigated properties of these congruences and established a system of
implications for the resulting quasivarieties.

Recall that $(\omega_t)_k=\pi$ is the least $E$-unitary congruence, and that $((\omega_k)_t)_k=\lambda$ is the least $E$-reflexive congruence.
An inverse semigroup $S$ is \emph{$E$-reflexive} if for any $x, y \in S$ and $e \in E_{\scriptscriptstyle{S}}$, $exy \in
E_{\scriptscriptstyle{S}}$ implies $eyx \in E_{\scriptscriptstyle{S}}$. Equivalently, $S$ is $E$-reflexive if and only if every $\eta$-class of
$S$, where $\eta$ denotes the least semilattice congruence, is $E$-unitary, i.e. $\eta$ is over $E$-unitary inverse semigroups. In this sense,
$E$-unitary inverse semigroups can be viewed as semigroups whose universal relation $\omega$ is over $E$-unitary inverse semigroups. There is
some relationship between the semigroups associated with the congruences $\beta_{n+2}$ and $\beta_n$ at the first few levels of the min
network. Dually, recall that $(\omega_k)_t=\nu$ is the least Clifford congruence, and $((\omega_t)_k)_t$ is the least $E\omega$-Clifford
congruence, or the least $\ker{\sigma}$-is-Clifford congruence. And Clifford semigroups can be regarded as $\ker{\omega}$-is-Clifford
semigroups in this sense. There is also a relationship between the semigroups associated with the congruences $\alpha_{n+2}$ and $\alpha_n$. We
wonder whether these patterns continue indefinitely.

Motivated by the symmetry observed above, our objective here is to obtain properties of the min network which highlights two series of inverse
semigroups, namely $\ker{\alpha_n}$-is-Clifford semigroups and $\beta_n$-is-over-$E$-unitary semigroups, and lead to characterizations of both
series. Finally we come to a similar but totally new system of implications. Although both of ours and Petrich -- Reilly's (\cite{network})
characterizations for the min network are inductive ones, Petrich -- Reilly focus on the the properties leading to expressions of
quasivarieties. The new characterization is based on all sorts of familiar, omnipresent relations, including special congruences, Green's
relations, $\mathcal{F}$ and $\mathcal{C}$-relations. It investigates the inner relations among these extremal congruences and the known
relations, which makes it possible to have more equivalent descriptions. Furthermore, the new characterization reflects symmetry in inverse
semigroups, where \textquoteleft\textquoteleft kernel\textquoteright\textquoteright~corresponds to \textquoteleft\textquoteleft
over\textquoteright\textquoteright~and \textquoteleft\textquoteleft Clifford\textquoteright\textquoteright~corresponds to
\textquoteleft\textquoteleft $E$-unitary\textquoteright\textquoteright.

In Section 1 we summarize notation and terminology to be used in the paper. In Section 2 we study $\ker{\alpha_n}$-is-Clifford semigroups,
$\beta_n$-is-over $E$-unitary semigroups and related congruences. A similar but symmetric system of implications for the quasivarieties induced
by the min network is established. The principal results for Section 3 are necessary and sufficient conditions for coincidences of certain
congruences.

\section{Preliminaries}

\emph{Throughout the entire paper, $S$ denotes an arbitrary inverse semigroup with semilattice $E_{\scriptscriptstyle{S}}$ of
idempotents.} When more than one semigroup is under discussion, $\theta(S)$ or $\theta(S/\rho)$ would be used to clarify the semigroup on which
the congruence is.

We shall use the notation and terminology of Howie \cite{fundamental} and Petrich \cite{inverse}, to which the reader is referred for basic
information and results on inverse semigroups. For an arbitrary inverse semigroup $S$, we denote by $E_{\scriptscriptstyle{S}}$ the semilattice
of its idempotents. The complete lattice of congruences on $S$ is denoted by $\mathcal{C}(S)$. For $\rho \in \mathcal{C}(S)$,
$\tr{\rho}=\rho|_{_{E_{\scriptscriptstyle{S}}}}$ is the \emph{trace} of $\rho$, and $\ker{\rho}=\{a \in S\,|\,a\, \rho\, e~ \text{for
some}~e \in E_{\scriptscriptstyle{S}}\}$ is the \emph{kernel} of $\rho$. The kernel of a congruence on an inverse semigroup is a normal inverse
subsemigroup. A congruence on an inverse semigroup is determined uniquely by its trace and kernel.

\begin{lemma}\textup{(\cite[Theorem 4.4]{congruence})}
 Let $\rho$ be a congruence on $S$. Then $$a~\rho~b \iff a^{-1}a\,\tr{\rho}\, b^{-1}b,~ab^{-1} \in \ker{\rho}.$$
\end{lemma}

For any $\rho$, $\theta \in \mathcal{C}(S)$, the relations $\mathcal{T}$ and $\mathcal{K}$ are defined as follows, $$\rho~\mathcal{T}~\theta
\iff \tr{\rho}=\tr{\theta}, \qquad \rho~\mathcal{K}~\theta \iff \ker{\rho}=\ker{\theta}.$$ The relation $\mathcal{T}$ is a complete congruence
on the lattice $\mathcal{C}(S)$, while $\mathcal{K}$ is an equivalence relation on
$\mathcal{C}(S)$. The equivalence class $\rho \mathcal{T}$ [resp. $\rho \mathcal{K}$] is an interval of $\mathcal{C}(S)$ with greatest and
least element to be denoted by $\rho^T$ [resp. $\rho^K$] and $\rho_t$ [resp. $\rho_k$], respectively.

\begin{lemma}\textup{(\cite[Theorem \@Roman3.2.5]{inverse})}
 For any congruence $\rho$ on $S$,
 \begin{align*}
  a~\rho^T~b &\iff a^{-1}ea\,\rho\,b^{-1}eb~\text{for all}~e \in E_{\scriptscriptstyle{S}},\\
  a~\rho_t~b &\iff ae=be~\text{for some}~e \in E_{\scriptscriptstyle{S}},~e\,\rho\,a^{-1}a\,\rho\,b^{-1}b.
 \end{align*}
\end{lemma}

On any inverse semigroup $S$, two relations $\mathcal{F}$ and $\mathcal{C}$ are defined by $$a\, \mathcal{F}\, b \iff a^{-1}b \in
E_{\scriptscriptstyle{S}}, \qquad a\, \mathcal{C}\, b \iff a^{-1}b,\, ab^{-1} \in E_{\scriptscriptstyle{S}}.$$

\begin{lemma}\label{min} \textup{(\cite[Theorem 6.2]{network})}
 For any congruence $\rho$ on an inverse semigroup $S$, $$\rho_t=(\rho \cap \mathcal{F})^*=(\rho \cap \mathcal{C})^*, \qquad \rho_k=(\rho \cap
 \mathcal{L})^*=(\rho \cap \mathcal{R})^*,$$ where $\xi^*$ denotes the least congruence on $S$ containing $\xi$.
\end{lemma}

$E_{\scriptscriptstyle{S}}\zeta$, the centralizer of $E_{\scriptscriptstyle{S}}$ in $S$, is defined by $$E_{\scriptscriptstyle{S}}\zeta=\{a \in
S\,|\,ae=ea~\text{for all}~e \in E_{\scriptscriptstyle{S}}\}.$$ $E_{\scriptscriptstyle{S}}\omega$, the closure of $E_{\scriptscriptstyle{S}}$
in $S$, is defined by $$E_{\scriptscriptstyle{S}}\omega=\{a \in S\,|\,a \ge e~\text{for some}~e \in E_{\scriptscriptstyle{S}}\},$$ where $\ge$
denotes the natural partial order on $S$ defined by $a
\le b \Leftrightarrow (\exists e \in E_{\scriptscriptstyle{S}})~a=eb \Leftrightarrow (\exists f \in E_{\scriptscriptstyle{S}})~a=bf$. A
semigroup which is a semilattice of groups is a \emph{Clifford semigroup}. Equivalently, $S$ is a Clifford semigroup if and only if $S$ is
regular and its idempotents lie in its centre. A semigroup $S$ is said to be \emph{$E$-unitary} if $ey=e$ for some $e \in
E_{\scriptscriptstyle{S}}$ implies that $y \in E_{\scriptscriptstyle{S}}$. Equivalently $S$ is $E$-unitary if and only if it satisfies the
implication $xy=x \Rightarrow y^2=y$. A subset $K$ of $S$ is \emph{full} if $E_{\scriptscriptstyle{S}} \subseteq K$. A congruence $\rho$
\emph{saturates} $K$ if $K$ is a union of $\rho$-classes.

Let $\mathcal{P}$ be a class of semigroups and $\rho \in \mathcal{C}(S)$. Then $\rho$ is \emph{over $\mathcal{P}$} if each $\rho$-class which
is a subsemigroup of $S$ belongs to $\mathcal{P}$. Also $\rho$ is a \emph{$\mathcal{P}$-congruence} if $S/\rho \in \mathcal{P}$. A congruence
$\rho$ on $S$ is \emph{idempotent separating} if $e^2=e$, $f^2=f$ and $e\, \rho\, f$ imply that $e=f$. Equivalently, $\rho$ is idempotent
separating if and only if $\rho \subseteq \mathcal{H}$. On the other hand, $\rho$ is \emph{idempotent pure} if
$\ker{\rho}=E_{\scriptscriptstyle{S}}$. Equivalently, $\rho$ is idempotent pure if and only if $\rho \subseteq \mathcal{C}$. We denote by
$\sigma$, $\eta$, $\mu$ and $\tau$ the least group, least semilattice, greatest idempotent separating and greatest idempotent pure congruences
on $S$, respectively. The equality and the universal relations on $S$ are denoted by $\varepsilon$ and $\omega$ respectively.

An inverse semigroup $S$ is \emph{fundamental} if $\varepsilon$ is the only congruence on $S$ contained in $\mathcal{H}$ (equivalently,
if $\mu=\varepsilon$). An inverse semigroup $S$ is $E$-\emph{disjunctive} if $\varepsilon$ is the only congruence on $S$ saturating
$E_{\scriptscriptstyle{S}}$ (equivalently, if $\tau=\varepsilon$).

Inverse semigroups the closure of whose set of idempotents is a Clifford semigroup were first studied by Billhardt \cite{closure}.

\begin{lemma}\textup{(\cite[Lemma 5]{closure})}
 Let $S$ be an inverse semigroup and $\sigma$ be the least group congruence on $S$. Then the following statements are equivalent.\\
(1) $E_S\omega$ is a Clifford semigroup;\\
(2) $[\,a\,\sigma\,b~\text{and}~a^{-1}a \leq b^{-1}b\,] \Rightarrow aa^{-1} \leq bb^{-1}$;\\
(3) $\sigma \cap \mathcal{L}=\sigma \cap \mathcal{R}$;\\
(4) $\sigma \cap \mathcal{L}$ is a congruence;\\
(5) $\sigma \cap \mathcal{R}$ is a congruence.
\end{lemma}

Properties of congruences obtained by starting with $\omega$ and successively forming $\rho_t$ and $\rho_k$ were first studied by Petrich --
Reilly \cite{network}.

\begin{definition}\textup{(\cite[Definition 5.1]{network})}\label{not}
 On $S$ we define inductively the following two sequences of congruences:
 \begin{eqnarray*}
  &&\alpha_0=\omega=\beta_0,\\
  &&\alpha_n=(\beta_{n-1})_t, \quad \beta_n=(\alpha_{n-1})_k \quad \text{for~} n \geqslant 1.
 \end{eqnarray*}
 We call the aggregate $\{\alpha_n, \beta_n\}_{n=0}^\infty$, together with the inclusion relation for congruences, the \emph{min network} of
 congruences on $S$.
\end{definition}

The min network is related to the following family of implications.
\begin{definition}\textup{(\cite[Definition 5.2]{network})}
 An inverse semigroup $S$ might satisfy one of the following implications:\\
 (A$_0$) $x=y$; (A$_1$) $x^{-1}x=y^{-1}y$; (A$_2$) $y \in E\zeta$;\\
 (A$_n$) $xy=x$, $x\,\beta_{n-3}\,y \Rightarrow y \in E\zeta$, $n \geqslant 3$;\\
 (B$_0$) $x=y$; (B$_1$) $y \in E$;\\
 (B$_n$) $xy=x$, $x\,\beta_{n-2}\,y \Rightarrow y \in E$, $n \geqslant 2$.
\end{definition}

The next few results develop some basic facts about the min network.

\begin{lemma}
 (1) \textup{(\cite[Proposition 5.3]{network})} For $n \geqslant 1$, we have $\alpha_{n-1} \cap \beta_{n-1}=\alpha_n \vee \beta_n$;\\
 (2) \textup{(\cite[Proposition 5.4]{network})} the min network, together with the intersections of corresponding pairs, is a sublattice of
 $\mathcal{C}(S)$.
\end{lemma}

The quotients $S/\omega_t$, $S/\omega_k$, $\cdots$, as $S$ runs over all inverse semigroups, form quasivarieties.

\begin{lemma}\textup{(\cite[Theorem 5.5]{network})}\label{fres}
 (1) $\alpha_n$ is the minimum congruence $\rho$ on $S$ such that $S/\rho$ satisfies (A$_n$);\\
 (2) $\beta_n$ is the minimum congruence $\rho$ on $S$ such that $S/\rho$ satisfies (B$_n$).
\end{lemma}

The first few levels of the min network are depicted in Figure \ref{original} (\cite{network}) together with some relationships and alternative
characterizations.

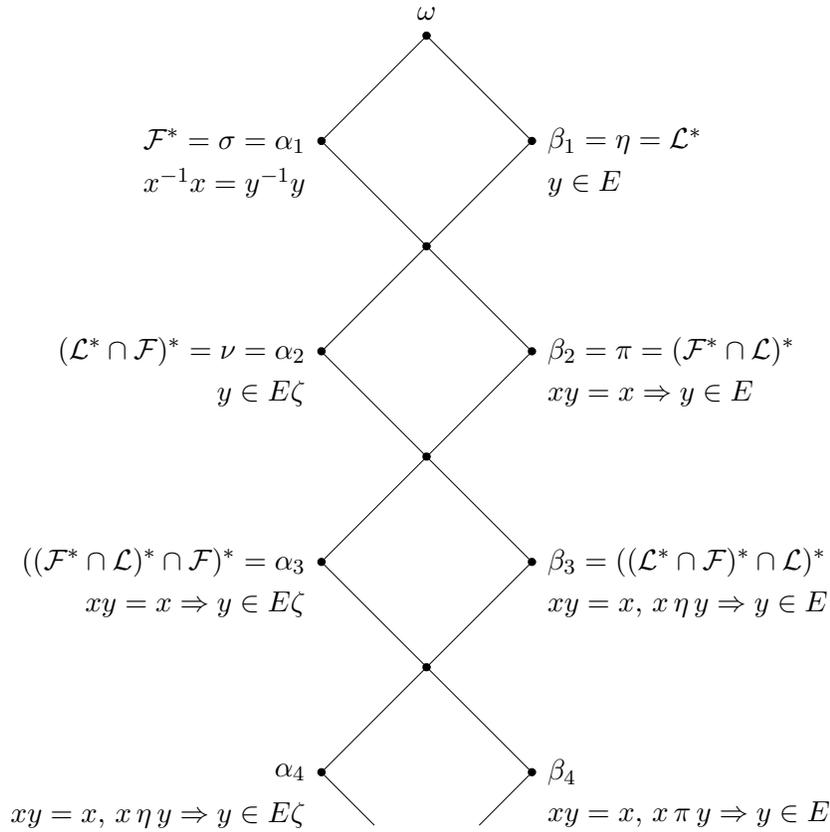
\begin{figure}[hbt]\label{original}
 \renewcommand*\figurename{Figure}
 \renewcommand*\captionlabeldelim{}
 \setlength{\unitlength}{0.7cm}
 \begin{center}
  \begin{picture}(4,16)
   \drawline(1,0)(0,1)(4,5)(0,9)(4,13)(2,15)(0,13)(4,9)(0,5)(4,1)(3,0)
   \allinethickness{0.7mm}
   \put(0,1){\circle*{0.06}} \put(0,5){\circle*{0.06}} \put(0,9){\circle*{0.06}} \put(0,13){\circle*{0.06}} \put(2,15){\circle*{0.06}}
   \put(2,11){\circle*{0.06}} \put(2,7){\circle*{0.06}} \put(2,3){\circle*{0.06}} \put(4,1){\circle*{0.06}} \put(4,5){\circle*{0.06}}
   \put(4,9){\circle*{0.06}} \put(4,13){\circle*{0.06}}
   \put(2,15.3){\makebox(0,0)[b]{$\omega$}}
   \put(-0.3,13){\makebox(0,0)[r]{$\mathcal{F}^*=\sigma=\alpha_1$}}
   \put(-0.3,9){\makebox(0,0)[r]{$(\mathcal{L}^* \cap \mathcal{F})^*=\nu=\alpha_2$}}
   \put(-0.3,5){\makebox(0,0)[r]{$((\mathcal{F}^* \cap \mathcal{L})^* \cap \mathcal{F})^*=\alpha_3$}}
   \put(-0.3,1){\makebox(0,0)[r]{$\alpha_4$}}
   \put(4.3,13){\makebox(0,0)[l]{$\beta_1=\eta=\mathcal{L}^*$}}
   \put(4.3,9){\makebox(0,0)[l]{$\beta_2=\pi=(\mathcal{F}^* \cap \mathcal{L})^*$}}
   \put(4.3,5){\makebox(0,0)[l]{$\beta_3=((\mathcal{L}^* \cap \mathcal{F})^* \cap \mathcal{L})^*$}}
   \put(4.3,1){\makebox(0,0)[l]{$\beta_4$}}
   \put(-0.3,12.2){\makebox(0,0)[r]{$x^{-1}x=y^{-1}y$}}
   \put(-0.3,8.2){\makebox(0,0)[r]{$y \in E\zeta$}}
   \put(-0.3,4.2){\makebox(0,0)[r]{$xy=x \Rightarrow y \in E\zeta$}}
   \put(-0.3,0.2){\makebox(0,0)[r]{$xy=x$, $x\,\eta\,y \Rightarrow y \in E\zeta$}}
   \put(4.3,12.2){\makebox(0,0)[l]{$y \in E$}}
   \put(4.3,8.2){\makebox(0,0)[l]{$xy=x \Rightarrow y \in E$}}
   \put(4.3,4.2){\makebox(0,0)[l]{$xy=x$, $x\,\eta\,y \Rightarrow y \in E$}}
   \put(4.3,0.2){\makebox(0,0)[l]{$xy=x$, $x\,\pi\,y \Rightarrow y \in E$}}
  \end{picture}
  \caption{\quad The first few levels of the min network} \label{origin1}
 \end{center}
\end{figure}

\section{Characterizations of $\alpha_{n+2}$ and $\beta_{n+2}$}

We will now develop characterizations of the congruences $\alpha_{n+2}$ and $\beta_{n+2}$ for any natural number $n$ on $S$. After defining
$\ker{\alpha_n}$-is-Clifford semigroups and $\beta_n$-is-over-$E$-unitary semigroups, we provide some equivalent conditions in terms of
implications as well as congruences. We then characterize $\ker{\alpha_n}$-is-Clifford congruences and $\beta_n$-is-over-$E$-unitary
congruences on an inverse semigroup $S$ and prove that they form a complete $\cap$-subsemilattice of the lattice of all congruences on $S$ with
least element $\alpha_{n+2}$ and $\beta_{n+2}$ respectively.

\begin{definition}
 An inverse semigroup for which $\ker{\alpha_n}$ is a Clifford [resp. $E$-reflexive] semigroup is called a \emph{$\ker{\alpha_n}$-is-Clifford}
 [resp. \emph{$\ker{\alpha_n}$-is-$E$-reflexive}] \emph{semigroup}. An inverse semigroup $S$ is called a \emph{$\beta_n$-is-over-$E$-unitary
 semigroup} if $e\beta_n$ is $E$-unitary for each $e \in E_{\scriptscriptstyle{S}}$. A congruence $\rho$ on $S$ is called a
 \emph{$\ker{\alpha_n}$-is-Clifford congruence} if $\ker{\alpha_n(S/\rho)}$ is a Clifford semigroup. A congruence $\rho$ on $S$ is called a
 \emph{$\beta_n$-is-over-$E$-unitary congruence} if $\beta_n(S/\rho)$ is over $E$-unitary semigroups.
\end{definition}

We shall need some auxiliary results first.

\begin{lemma}\label{equivalent}
 For $n \geqslant 2$, semigroups satisfying (B$_n$) are exactly $\beta_{n-2}$-is-over-$E$-unitary semigroups.
\end{lemma}
\begin{proof}
 First suppose that $S$ satisfies (B$_n$) and let $x$, $y \in e\beta_{n-2}$ with $xy=x$. Then it is clear from the assumption that $y \in E$,
 that is, $e\beta_{n-2}$ is $E$-unitary.

 Conversely, suppose that $S$ is a $\beta_{n-2}$-is-over-$E$-unitary semigroup, and let $xy=x$ with $x\,\beta_{n-2}\,y$. Then
 $x=xy\,\beta_{n-2}\,y^2\,\beta_{n-2}\,x^2$ and $x\beta_{n-2} \in E(S/\beta_{n-2})$. Using our assumption we find that $x\beta_{n-2}$ is
 $E$-unitary whence $y \in E$.
\end{proof}

\begin{remark}\label{beta}
 Lemma \ref{fres} and Lemma \ref{equivalent} show that $\beta_n$ is the least $\beta_{n-2}$-is-over-$E$-unitary congruence.
\end{remark}

Let $\mathcal{A}_n$ denote the set of all congruences $\gamma$ on $S$ such that the kernel of $\alpha_n(S/\gamma)$ is a Clifford semigroup. Let
$\mathcal{B}_n$ denote the set of all congruences $\theta$ on $S$ such that $\beta_n(S/\theta)$ is over $E$-unitary semigroups.

\begin{lemma}\label{referee}
 For $n \geqslant 0$, $\mathcal{A}_n$ and $\mathcal{B}_n$ have least elements.
\end{lemma}
\begin{proof}
 Since the kernel of $\alpha_n(S/\omega)$ is trivial, it follows that $\omega \in \mathcal{A}_n$ so that $\mathcal{A}_n \neq \emptyset$.

 Suppose that $\mathcal{G}$ is a nonempty family of $\ker{\alpha_n}$-is-Clifford congruences. It follows from Lemma \ref{fres} that the semigroups $(S/\rho)/(\alpha_n(S/\rho))$ ($\rho \in \mathcal{G}$) all satisfy the implications in (A$_n$), hence so also does their direct product $\prod\limits_{\rho \in \mathcal{G}} (S/\rho)/(\alpha_n(S/\rho))$ as well as any subdirect product of $\prod\limits_{\rho \in \mathcal{G}} (S/\rho)/(\alpha_n(S/\rho))$.

 Let $\gamma$ denote the product of congruences $\prod\limits_{\rho \in \mathcal{G}} \alpha_n(S/\rho)$. Then $\gamma$ is a congruence on $\prod\limits_{\rho \in \mathcal{G}} S/\rho$. Now $S/(\bigcap\limits_{\rho \in \mathcal{G}} \rho)$ is (isomorphic to) a subdirect product of $\prod\limits_{\rho \in \mathcal{G}} S/\rho$ and therefore $\gamma$ induces a congruence on $S/(\bigcap\limits_{\rho \in \mathcal{G}} \rho)$. In addition, we have $$(\prod\limits_{\rho \in \mathcal{G}} S/\rho)/(\prod\limits_{\rho \in \mathcal{G}} \alpha_n(S/\rho)) \simeq \prod\limits_{\rho \in \mathcal{G}} (S/\rho)/(\alpha_n(S/\rho))$$ which also satisfies (A$_n$). But $\alpha_n(\prod\limits_{\rho \in \mathcal{G}} S/\rho)$ is the least such congruence by Lemma \ref{fres} and therefore $\alpha_n(\prod\limits_{\rho \in \mathcal{G}} S/\rho) \subseteq \prod\limits_{\rho \in \mathcal{G}} \alpha_n(S/\rho)$ so that $\ker{\alpha_n(\prod\limits_{\rho \in \mathcal{G}} S/\rho)} \subseteq \ker{(\prod\limits_{\rho \in \mathcal{G}} \alpha_n(S/\rho))}=\prod\limits_{\rho \in \mathcal{G}} \ker{(\alpha_n(S/\rho))}$. However, the kernels of $\alpha_n(S/\rho)$ ($\rho \in \mathcal{G}$) are Clifford semigroups. This implies that the kernel of $\prod\limits_{\rho \in \mathcal{G}} \alpha_n(S/\rho)$ is also a Clifford semigroup and therefore so also is the kernel of $\alpha_n(S/(\bigcap\limits_{\rho \in \mathcal{G}} \rho))$. Therefore $\bigcap\limits_{\rho \in \mathcal{G}} \rho \in \mathcal{A}_n$. In other words, the set of congruences on $S$ for which the quotient is a $\ker{\alpha_n}$-is-Clifford semigroup is closed under arbitrary intersections. Therefore there exists a least such congruence, that is, $\mathcal{A}_n$ has a least element. A similar argument establishes the assertion concerning $\mathcal{B}_n$ and so the proof is complete.
\end{proof}

We are now ready for characterizations of $\ker{\alpha_n}$-is-Clifford inverse semigroups.

\begin{proposition}\label{kercliff}
 For $n \geqslant 1$, the following conditions on an inverse semigroup $S$ are equivalent.\\
 (1) $S$ is a $\ker{\alpha_n}$-is-Clifford semigroup;\\
 (2) $[\,a\,\alpha_n\,b~\text{and}~a^{-1}a \leq b^{-1}b\,] \Rightarrow aa^{-1} \leq bb^{-1}$;\\
 (3) $\alpha_n \cap \mathcal{L}=\alpha_n \cap \mathcal{R}$;\\
 (4) $\alpha_n \cap \mathcal{L}$ is a congruence;\\
 (5) $\alpha_n \cap \mathcal{R}$ is a congruence;\\
 (6) $\alpha_n \cap \mathcal{L}=\alpha_n \cap \mu$;\\
 (7) there exists an idempotent separating $\beta_{n-1}$-is-over-$E$-unitary congruence on $S$;\\
 (8) $\beta_{n+1} \subseteq \mu$;\\
 (9) $(\beta_{n+1})_t=\varepsilon$;\\
 (10) $\beta_{n+1} \cap \mathcal{F}=\varepsilon$;\\
 (11) $\ker{\alpha_n} \subseteq E_{\scriptscriptstyle{S}}\zeta$;\\
 (12) $S$ satisfies the implication $xy=x$, $x^{-1}x\,\alpha_n\,yy^{-1} \Rightarrow y \in E_{\scriptscriptstyle{S}}\zeta$.
\end{proposition}
\begin{proof}
 $(1) \Rightarrow (2)$. Let $a\,\alpha_n\,b$ and $a^{-1}a \leq b^{-1}b$. Then $ba^{-1} \in \ker{\alpha_n}$ whence it also follows that $a^{-1}
 \leq b^{-1}ba^{-1}=b^{-1}(bb^{-1})(ba^{-1})=b^{-1}(ba^{-1})(bb^{-1})$. Thus $aa^{-1} \leq (ab^{-1})(ba^{-1})(bb^{-1}) \leq bb^{-1}$.

 $(2) \Rightarrow (3)$. From the hypothesis, we have
 \begin{eqnarray*}
  a\,(\alpha_n \cap \mathcal{L})\,b &\iff& a^{-1}a=b^{-1}b \text{~and~} a\,\alpha_n\,b\\
  &\iff& aa^{-1}=bb^{-1} \text{~and~} a\,\alpha_n\,b\\
  &\iff& a\,(\alpha_n \cap \mathcal{R})\,b.
 \end{eqnarray*}

 $(3) \Rightarrow (4)$. Obvious, since $\mathcal{L}$ is a right and $\mathcal{R}$ is a left congruence.

 $(4) \Rightarrow (1)$. Here we have
 \begin{eqnarray*}
  a \in \ker{\alpha_n} &\Longrightarrow& a^{-1}a\,(\alpha_n \cap \mathcal{L})\,a\\
  &\Longrightarrow& aa^{-1}a^{-1}a\,(\alpha_n \cap \mathcal{L})\,aa^{-1}a=a \qquad \text{since~} \alpha_n \cap \mathcal{L} \text{~is a
  congruence}\\
  &\Longrightarrow& aa^{-1}a^{-1}a=a^{-1}a,
 \end{eqnarray*}
 and
 \begin{eqnarray*}
  a \in \ker{\alpha_n} &\Longrightarrow& a^{-1} \in \ker{\alpha_n}\\
  &\Longrightarrow& aa^{-1}\,(\alpha_n \cap \mathcal{L})\,a^{-1}\\
  &\Longrightarrow& a^{-1}aaa^{-1}\,(\alpha_n \cap \mathcal{L})\,a^{-1}aa^{-1}=a^{-1} \qquad \text{since~} \alpha_n \cap \mathcal{L} \text{~is
  a congruence}\\
  &\Longrightarrow& a^{-1}aaa^{-1}=aa^{-1}.
 \end{eqnarray*}
 Therefore we have $aa^{-1}=a^{-1}a$. It follows that $\ker{\alpha_n}$ is a Clifford semigroup.

 $(3) \Rightarrow (5) \Rightarrow (1)$. The proof is dual to that for $(3) \Rightarrow (4) \Rightarrow (1)$ and is omitted.

 $(4) \Rightarrow (6)$. On the one hand, $\alpha_n \cap \mathcal{L} \in \mathcal{C}(S)$ and $\alpha_n \cap \mathcal{L} \subseteq \mathcal{L}$
 give that $\alpha_n \cap \mathcal{L}$ is idempotent separating. Hence $\alpha_n \cap \mathcal{L} \subseteq \mu$ so that $\alpha_n \cap
 \mathcal{L} \subseteq \alpha_n \cap \mu$. On the other hand, $\mu \subseteq \mathcal{L}$ gives $\alpha_n \cap \mu \subseteq \alpha_n \cap
 \mathcal{L}$. Consequently, $\alpha_n \cap \mathcal{L}=\alpha_n \cap \mu$, as required.

 $(6) \Rightarrow (7)$. From $\alpha_n \cap \mathcal{L}=\alpha_n \cap \mu$ it follows that $\alpha_n \cap \mathcal{L}$ is a congruence and thus
 also that the $\beta_{n-1}$-is-over-$E$-unitary congruence $\beta_{n+1}=(\alpha_n)_k=(\alpha_n \cap \mathcal{L})^*=\alpha_n \cap \mathcal{L}
 \subseteq \mathcal{L}$ is idempotent separating.

 $(7) \Rightarrow (8)$. Assume that $\rho$ is an idempotent separating congruence such that $\beta_{n-1}(S/\rho)$ is over $E$-unitary
 semigroups. By Remark \ref{beta}, $\beta_{n+1}$ is the least such congruence. Therefore $\beta_{n+1} \subseteq \rho \subseteq \mu$.

 $(8) \Rightarrow (9)$. Since $\beta_{n+1} \subseteq \mu$, we have $(\beta_{n+1})_t \subseteq \mu_t=\varepsilon$ and thus
 $(\beta_{n+1})_t=\varepsilon$.

 $(9) \Rightarrow (10)$. It follows directly from Lemma \ref{min}.

 $(10) \Rightarrow (7)$. If $\beta_{n+1} \cap \mathcal{F}=\varepsilon$, then by Lemma \ref{min} $\beta_{n+1}$ is idempotent separating. By
 Remark \ref{beta}, $\beta_{n+1}$ is a $\beta_{n-1}$-is-over-$E$-unitary congruence.

 $(7) \Rightarrow (4)$. Since $(7) \Rightarrow (8)$, we know that $\beta_{n+1} \subseteq \rho \subseteq \mu$. By \cite[Proposition
 III.3.2]{inverse}, $\mu \subseteq \mathcal{L}$. Hence Definition \ref{not} and Lemma \ref{min} give that $(\alpha_n \cap
 \mathcal{L})^*=(\alpha_n)_k=\beta_{n+1} \subseteq \mathcal{L}$. Thus, with $\alpha_n \cap \mathcal{L} \subseteq \alpha_n$, we have $(\alpha_n
 \cap \mathcal{L})^* \subseteq \alpha_n^* = \alpha_n$ and $(\alpha_n \cap \mathcal{L})^*=\alpha_n \cap \mathcal{L}$ so that $\alpha_n \cap
 \mathcal{L}$ is a congruence.

 $(1) \Rightarrow (11)$. Let $a \in \ker{\alpha_n}$. By the fact that $\ker{\alpha_n}$ is a full inverse subsemigroup and the assumption that
 $\ker{\alpha_n}$ is a Clifford semigroup, we find that $ae=ea$ for all $e \in E_{\scriptscriptstyle{S}}$, and thus $a \in
 E_{\scriptscriptstyle{S}}\zeta$.

 $(11) \Rightarrow (12)$. Let $xy=x$ and $x^{-1}x\,\alpha_n\,yy^{-1}$. Then $y\,\alpha_n\,x^{-1}xy=x^{-1}x$. But $\ker{\alpha_n} \subseteq
 E_{\scriptscriptstyle{S}}\zeta$ and so $y \in E_{\scriptscriptstyle{S}}\zeta$.

 $(12) \Rightarrow (1)$. Let $a \in \ker{\alpha_n}$. By the dual of \cite[Notation III.2.4]{inverse} and \cite[Exercise
 III.2.14(iii)]{inverse}, $ea=e$ for some $e \in E_{\scriptscriptstyle{S}}$ with $e\,\beta_{n-1}\,aa^{-1}$. Notice that
 $\tr{\beta_{n-1}}=\tr{\alpha_n}$. We have $e^{-1}e=e\,\alpha_n\,aa^{-1}$ and therefore $a \in E_{\scriptscriptstyle{S}}\zeta$ by assumption.
 This together with the fact that $\ker{\alpha_n}$ is a full inverse subsemigroup gives that $S$ is a $\ker{\alpha_n}$-is-Clifford semigroup.
\end{proof}

An important property of $\ker{\alpha_n}$-is-Clifford semigroups is contained in the following proposition.

\begin{proposition}\label{class}
 Let $S$ be an inverse semigroup and $n \geqslant 2$. If $\ker{\alpha_{n-1}} \cap N$ is a Clifford subsemigroup for every $\eta$-class $N$ of
 $S$, then $\ker{\alpha_n}$ is a Clifford semigroup.
\end{proposition}
\begin{proof}
 Let $a \in \ker{\alpha_n}$ and $f \in E_{\scriptscriptstyle{S}}$. Since $a\,\eta\,a^{-1}a$, we have $af\,\eta\,a^{-1}af$. Further,
 $a\,\alpha_n\,a^{-1}a$ gives $af\,\alpha_n\,a^{-1}af$, whence $af$, $a^{-1}af \in \ker{\alpha_n} \subseteq \ker{\alpha_{n-1}}$. We
 consequently have $(af)(a^{-1}af)=(a^{-1}af)(af)$ since $\ker{\alpha_{n-1}} \cap (a^{-1}af)\,\eta$ is a Clifford subsemigroup of $S$. Notice
 that $(af)(a^{-1}af)=a(a^{-1}af)=af$ and $(a^{-1}af)(af)=(fa^{-1}a)(af)$. It follows that $af=fa^{-1}aaf$ and $faf=f(fa^{-1}aaf)=af$. But
 $a\,\eta\,aa^{-1}$ and so $fa\,\eta\,faa^{-1}$. Again, $a\,\alpha_n\,aa^{-1}$ gives $fa\,\alpha_n\,faa^{-1}$, whence $fa$, $faa^{-1} \in
 \ker{\alpha_n} \subseteq \ker{\alpha_{n-1}}$. Therefore we have $(fa)(faa^{-1})=(faa^{-1})(fa)$ by assumption. It is clear from
 $(fa)(faa^{-1})=faaa^{-1}f$ and $(faa^{-1})(fa)=f(faa^{-1})a=fa$ that $fa=faaa^{-1}f$ and $faf=(faaa^{-1}f)f=fa$. We conclude that $fa=faf=af$
 and that $\ker{\alpha_n}$ is a Clifford semigroup.
\end{proof}

\begin{remark}
 Proposition \ref{class} presents a response to the problem in \cite{feng}.
\end{remark}

For a given congruence $\rho$, an exactly parallel argument to Lemma \ref{referee}'s establishes that the least $\beta_n$-is-over-$E$-unitary
congruence containing $\rho$ exists. Denote it by $(\beta_{n+2})_\rho$. The next result characterizes $\ker{\alpha_n}$-is-Clifford congruences
in terms of more familiar notions.

\begin{proposition}\label{kercliffcon}
 For $n \geqslant 1$, the following statements concerning a congruence $\rho$ on an inverse semigroup $S$ are equivalent.\\
 (1) $\rho$ is a $\ker{\alpha_n}$-is-Clifford congruence;\\
 (2) $(\beta_{n+1})_{\rho} \subseteq \rho^T$, where $(\beta_{n+1})_{\rho}$ is the least $\beta_{n-1}$-is-over-$E$-unitary congruence on $S$
 containing $\rho$;\\
 (3) $\tr{(\beta_{n+1})_{\rho}}=\tr{\rho}$.
\end{proposition}
\begin{proof}
 $(1) \Rightarrow (2)$. Clearly $(S/\rho)/((\beta_{n+1})_\rho/\rho) \simeq S/(\beta_{n+1})_\rho$. Thus $(\beta_{n+1})_\rho/\rho$ is a
 $\beta_{n-1}$-is-over-$E$-unitary congruence on $S/\rho$. If $\theta/\rho$ is a  $\beta_{n-1}$-is-over-$E$-unitary congruence on $S/\rho$ with
 $\rho \subseteq \theta$, then $S/\theta \simeq (S/\rho)/(\theta/\rho)$ which implies that $\theta$ is a  $\beta_{n-1}$-is-over-$E$-unitary
 congruence on $S$. Hence $(\beta_{n+1})_\rho \subseteq \theta$ and $(\beta_{n+1})_\rho/\rho \subseteq \theta/\rho$. Consequently
 $(\beta_{n+1})_\rho/\rho$ is the least $\beta_{n-1}$-is-over-$E$-unitary congruence on $S/\rho$ whence
 $\beta_{n+1}(S/\rho)=(\beta_{n+1})_\rho/\rho$.

 If $S/\rho$ is a $\ker{\alpha_n}$-Clifford semigroup, then $(\beta_{n+1})_{\rho}/\rho=\beta_{n+1}(S/\rho) \subseteq \mu(S/\rho)=\rho^T/\rho$,
 and thus $(\beta_{n+1})_{\rho} \subseteq \rho^T$.

 $(2) \Rightarrow (3)$. Since $\rho \subseteq (\beta_{n+1})_{\rho} \subseteq \rho^T$, we have $\tr{\rho} \subseteq \tr{(\beta_{n+1})_{\rho}}
 \subseteq \tr{\rho^T}=\tr{\rho}$, which implies $\tr{(\beta_{n+1})_{\rho}}=\tr{\rho}$.

 $(3) \Rightarrow (1)$. $\tr{(\beta_{n+1})_{\rho}}=\tr{\rho}$ implies $(\beta_{n+1})_{\rho} \subseteq \rho^T$. Since
 $\tr{\rho}=\tr{(\beta_{n+1})_{\rho}}$, $(\beta_{n+1})_{\rho}/\rho$ is an idempotent separating congruence on $S/\rho$, which gives that
 $S/\rho$ is a $\ker{\alpha_n(S/\rho)}$-is-Clifford semigroup by Proposition \ref{kercliff}. This completes the proof that $\rho$ is a
 $\ker{\alpha_n}$-is-Clifford congruence.
\end{proof}

Remind that $\mathcal{A}_n$ is the set of all congruences $\gamma$ on an inverse semigroup $S$ such that the kernel of $\alpha_n(S/\gamma)$ is
a Clifford semigroup. Equivalently, $\mathcal{A}_n$ is the set of all $\ker{\alpha_n}$-is-Clifford congruences on $S$ ordered by inclusion.

\begin{theorem} \label{kccex}
 Let $S$ be an inverse semigroup.\\
 (1) $\mathcal{A}_n$ is a complete $\cap$-subsemilattice of $\mathcal{C}(S)$ whose least element is $\alpha_{n+2}=(\beta_{n+1})_t=(\beta_{n+1}
 \cap \mathcal{F})^*$ and greatest element is $\omega$;\\
 (2) the interval $[\alpha_{n+2},\,\beta_{n+1}]$ is a complete sublattice of $\mathcal{A}_n$.
\end{theorem}
\begin{proof}
 (1) It follows directly from Lemma \ref{referee} that $\mathcal{A}_n$ is a complete $\cap$-subsemilattice of $\mathcal{C}(S)$.

 Since $(\beta_{n+1})_{(\beta_{n+1})_t}=\beta_{n+1} \subseteq (\beta_{n+1})^T=((\beta_{n+1})_t)^T$, we have by Proposition \ref{kercliffcon}
 that $(\beta_{n+1})_t$ is a $\ker{\alpha_n}$-is-Clifford congruence. If $\rho$ is a $\ker{\alpha_n}$-is-Clifford congruence, then $\beta_{n+1}
 \subseteq (\beta_{n+1})_\rho \subseteq \rho^T$ and $(\beta_{n+1})_t \subseteq (\rho^T)_t=\rho_t \subseteq \rho$. This proves that
 $\alpha_{n+2}=(\beta_{n+1})_t$ is the least $\ker{\alpha_n}$-is-Clifford congruence.

 (2) If $\rho \in [\alpha_{n+2},\,\beta_{n+1}]$, then $\tr{\rho}=\tr{\alpha_{n+2}}=\tr{\beta_{n+1}}$. But then $(\beta_{n+1})_\rho=\beta_{n+1}
 \subseteq \rho^T$, and thus Proposition \ref{kercliffcon} gives that $\rho$ is a $\ker{\alpha_n}$-is-Clifford congruence.

 Let $\mathcal{A}$ be a non-empty family of congruences on $S$ such that $\rho \in [\alpha_{n+2},\,\beta_{n+1}]$ for every $\rho \in
 \mathcal{A}$, then $\underset{\scriptscriptstyle{\rho \in \mathcal{A}}}{\bigcap}\rho$, $\underset{\scriptscriptstyle{\rho \in
 \mathcal{A}}}{\bigvee}\rho \in [\alpha_{n+2},\,\beta_{n+1}]$, and so $\underset{\scriptscriptstyle{\rho \in \mathcal{A}}}{\bigcap}\rho$,
 $\underset{\scriptscriptstyle{\rho \in \mathcal{A}}}{\bigvee}\rho$ are $\ker{\alpha_n}$-is-Clifford congruences by what was proved earlier,
 which completes the proof of the assertion.
\end{proof}

We now turn to characterizations of $\beta_{n+2}$. Compare the following result with Proposition \ref{kercliff}.

\begin{proposition}\label{boeu}
 For $n \geqslant 1$, the following conditions on an inverse semigroup $S$ are equivalent.\\
 (1) $S$ is a $\beta_n$-is-over-$E$-unitary semigroup;\\
 (2) $\beta_{n} \cap \mathcal{F}$ is a congruence;\\
 (3) $\beta_{n} \cap \mathcal{C}$ is a congruence;\\
 (4) $\beta_{n} \cap \mathcal{F}=\beta_{n} \cap \tau$;\\
 (5) $\beta_{n} \cap \mathcal{C}=\beta_{n} \cap \tau$;\\
 (6) there exists an idempotent pure $\ker{\alpha_{n-1}}$-is-Clifford congruence on $S$;\\
 (7) $\alpha_{n+1} \subseteq \tau$;\\
 (8) $(\alpha_{n+1})_k=\varepsilon$;\\
 (9) $\tr{\beta_n} \subseteq \tr{\tau}$;\\
 (10) $S$ satisfies the implication $xy=x$, $x^{-1}x\,\alpha_{n+1}\,yy^{-1} \Rightarrow y \in E_{\scriptscriptstyle{S}}$;\\
 (11) $\alpha_{n+1} \cap \mathcal{L}=\varepsilon$.
\end{proposition}
\begin{proof}
 $(7) \Rightarrow (8)$. It follows from $\alpha_{n+1} \subseteq \tau$ that $(\alpha_{n+1})_k \subseteq \tau_k=\varepsilon$, and hence that
 $(\alpha_{n+1})_k=\varepsilon$.

 $(8) \Rightarrow (7)$. By $(\alpha_{n+1})_k=\varepsilon$, we have $\ker{\alpha_{n+1}}=\ker{(\alpha_{n+1})_k}=E_{\scriptscriptstyle{S}}$. Hence
 $\alpha_{n+1}$ is idempotent pure so that $\alpha_{n+1} \subseteq \tau$.

 $(7) \Rightarrow (6)$. The hypothesis implies that the $\ker{\alpha_{n-1}}$-is-Clifford congruence $\alpha_{n+1}$ is idempotent pure.

 $(6) \Rightarrow (3)$. Assume that $\rho$ is an idempotent pure $\ker{\alpha_{n-1}}$-Clifford congruence. Then $\alpha_{n+1} \subseteq \rho
 \subseteq \tau \subseteq \mathcal{C}$ so that $(\beta_n \cap \mathcal{C})^*=\alpha_{n+1} \subseteq \mathcal{C}$. Also by $(\beta_n \cap
 \mathcal{C})^* \subseteq \beta_n$, $(\beta_n \cap \mathcal{C})^* \subseteq \beta_n \cap \mathcal{C}$ and thus $(\beta_n \cap
 \mathcal{C})^*=\beta_n \cap \mathcal{C}$, which implies that $\beta_n \cap \mathcal{C}$ is a congruence.

 $(3) \Rightarrow (9)$. If $\beta_{n} \cap\,\mathcal{C}$ is a congruence, then $\beta_{n} \cap\,\mathcal{C}$ is idempotent pure since
 $\beta_{n} \cap\,\mathcal{C} \subseteq \mathcal{C}$. Hence $\beta_{n} \cap\,\mathcal{C} \subseteq \tau$ and $\beta_{n} \cap\,\mathcal{C}
 \subseteq \beta_{n} \cap \tau$. Therefore $\beta_{n} \cap\,\mathcal{C}=\beta_{n} \cap \tau$ from the fact that $\tau \subseteq \mathcal{C}$.

 Let $e$, $f \in E_{\scriptscriptstyle{S}}$ with $e\,\beta_{n}\,f$. Then $e\,(\beta_{n} \cap\,\mathcal{C})\,f$ since any two idempotents is
 $\mathcal{C}$-related on inverse semigroups. By $\beta_{n} \cap\,\mathcal{C}=\beta_{n} \cap \tau$, we get $e\,\tau\,f$, as required.

 $(9) \Rightarrow (7)$. Suppose that $a \in \ker{\alpha_{n+1}}=\ker{(\beta_{n})_t}$. Then by \cite[Exercises
 \uppercase\expandafter{\romannumeral 3.2.14 (\romannumeral 3)}]{inverse} there exists $e \in E_{\scriptscriptstyle{S}}$ such that $ae=e$ and
 $e\,\beta_{n}\,a^{-1}a$, and thus $e\,\tau\,a^{-1}a$ by assumption. Hence $e=ae\,\tau\,a(a^{-1}a)=a$ which gives $a \in
 E_{\scriptscriptstyle{S}}$.

 $(3) \Rightarrow (2)$. If $\beta_{n} \cap \mathcal{C}$ is a congruence, then $\beta_{n} \cap \mathcal{F} \subseteq (\beta_{n} \cap
 \mathcal{F})^*=(\beta_{n} \cap \mathcal{C})^*=\beta_{n} \cap \mathcal{C} \subseteq \beta_{n} \cap \mathcal{F}$, and thus $\beta_{n} \cap
 \mathcal{F}=\beta_{n} \cap \mathcal{C}$ is a congruence.

 $(2) \Rightarrow (4)$. On the one hand, $\beta_{n} \cap \mathcal{F} \in \mathcal{C}(S)$ and $\beta_{n} \cap \mathcal{F} \subseteq \mathcal{F}$
 give that $\beta_{n} \cap \mathcal{F}$ is idempotent pure. Hence $\beta_{n} \cap \mathcal{F} \subseteq \tau$ so that $\beta_{n} \cap
 \mathcal{F} \subseteq \beta_{n} \cap \tau$. On the other hand, $\tau \subseteq \mathcal{F}$ gives $\beta_{n} \cap \tau \subseteq \beta_{n}
 \cap \mathcal{F}$. Consequently, $\beta_{n} \cap \mathcal{F}=\beta_{n} \cap \tau$, as required.

 $(4) \Rightarrow (5)$. Assume that $\beta_{n} \cap \mathcal{F}=\beta_{n} \cap \tau$. Then $\beta_{n} \cap \mathcal{C} \subseteq \beta_{n} \cap
 \mathcal{F}=\beta_{n} \cap \tau \subseteq \beta_{n} \cap \mathcal{C}$, which gives that $\beta_{n} \cap \mathcal{C}=\beta_{n} \cap \tau$.

 $(5) \Rightarrow (7)$. It follows directly from the hypothesis that $\alpha_{n+1}=(\beta_{n})_t=(\beta_{n} \cap \mathcal{C})^*=(\beta_{n} \cap
 \tau)^*=\beta_{n} \cap \tau \subseteq \tau$.

 $(3) \Rightarrow (11)$. Since $\beta_{n} \cap \mathcal{C}$ is a congruence, we have that $\alpha_{n+1}=(\beta_{n})_t=(\beta_{n} \cap
 \mathcal{C})^*=\beta_{n} \cap \mathcal{C}$, and thus $\alpha_{n+1} \cap \mathcal{L}=\beta_{n} \cap \mathcal{C} \cap \mathcal{L}=\beta_{n} \cap
 \varepsilon=\varepsilon$.

 $(11) \Rightarrow (7)$. Since $\alpha_{n+1} \cap \mathcal{L}=\varepsilon$, by \cite[Proposition \@Roman3.4.2]{inverse} we have that
 $\alpha_{n+1}$ is idempotent pure and thus $\alpha_{n+1} \subseteq \tau$.

 $(7) \Rightarrow (1)$. Let $a \in e\beta_{n}$ and $a \in E_{\scriptscriptstyle{e\beta_{n}}}\omega$. Then $a=fg$ for some $f$, $g \in
 E_{\scriptscriptstyle{e\beta_{n}}}$ so that $ag=fg$ and $g\,\beta_{n}\,a^{-1}a\,\beta_{n}\,f$. Thus $a\,(\beta_{n})_t\,f$ which implies that
 $a\,\alpha_{n+1}\,f$. But $\alpha_{n+1} \subseteq \tau$ which yields $a \in E$.

 $(1) \Rightarrow (10)$. Let $x,\,y \in S$ be such that $xy=x$ and $x^{-1}x\,\alpha_{n+1}\,yy^{-1}$. Then $x^{-1}x\,\beta_n\,yy^{-1}$ since
 $\tr{\alpha_{n+1}}=\tr{\beta_n}$. Hence $x^{-1}x=x^{-1}xy\,\beta_n\,yy^{-1}y=y$, which together with $x^{-1}xy=x^{-1}x$ implies $y \in
 E_{\scriptscriptstyle{S}}$ by assumption.

 $(10) \Rightarrow (7)$. Assume that $a\,(\beta_{n})_t\,e$ for some $e \in E_{\scriptscriptstyle{S}}$. Then there exists $f \in
 E_{\scriptscriptstyle{S}}$ such that $fa=fe$ and $f\,\beta_{n}\,aa^{-1}\,\beta_{n}\,e$, which implies that $(fe)a=e(fa)=e(fe)=fe$ and
 $fe=fa\,\beta_{n}\,(aa^{-1})a=a$. Therefore $fe\,\beta_n\,aa^{-1}$ whence $fe\,\alpha_{n+1}\,aa^{-1}$, since $\tr{\beta_n}=\tr{\alpha_{n+1}}$.
 The hypothesis yields $a \in E_{\scriptscriptstyle{S}}$ and thus $\ker{(\beta_{n})_t}=E_{\scriptscriptstyle{S}}$ so that
 $\alpha_{n+1}=(\beta_{n})_t \subseteq \tau$.
\end{proof}

The next proposition illustrates this class of inverse semigroups.

\begin{proposition}\label{ker}
 Let $S$ be a $\beta_n$-is-over-$E$-unitary inverse semigroup and $n \geqslant 1$. Then $S$ is a $\ker{\alpha_{n-1}}$-is-$E$-reflexive
 semigroup.
\end{proposition}
\begin{proof}
 Let $x, y \in \ker{\alpha_{n-1}}$, $e \in E_{\scriptscriptstyle{S}}$ be such that $exy \in E_{\scriptscriptstyle{S}}$. Then
 $$(exy)^{-1}(yexyy^{-1})=(exy)^{-1}(y(exy)y^{-1}) \in E_{\scriptscriptstyle{S}},$$ $$(exy)(yexyy^{-1})^{-1}=(exy)(y(exy)^{-1}y^{-1}) \in
 E_{\scriptscriptstyle{S}};$$ $$(eyx)^{-1}(yexyy^{-1})=(x^{-1}((y^{-1}ey)e)x)(yy^{-1}) \in E_{\scriptscriptstyle{S}},$$
 $$(eyx)(yexyy^{-1})^{-1}=e(y((x(yy^{-1})x^{-1})e)y^{-1}) \in E_{\scriptscriptstyle{S}}.$$ Hence $exy\, \mathcal{C}\, yexyy^{-1}$, $eyx\,
 \mathcal{C}\, yexyy^{-1}$. But $x, y \in \ker{\alpha_{n-1}}=\ker{\beta_{n}}$, so
 $$exy\,\beta_{n}\,exx^{-1}yy^{-1}=yy^{-1}exx^{-1}yy^{-1}\,\beta_{n}\,yexyy^{-1},$$
 $$eyx\,\beta_{n}\,eyy^{-1}xx^{-1}=yy^{-1}exx^{-1}yy^{-1}\,\beta_{n}\,yexyy^{-1}.$$ Hence $exy\,(\beta_{n} \cap \mathcal{C})\,yexyy^{-1}$,
 $eyx\,(\beta_{n} \cap \mathcal{C})\,yexyy^{-1}$. Since $\beta_{n} \cap\,\mathcal{C}$ is a congruence by Proposition \ref{boeu}, it is also an
 equivalence relation. So $exy\,(\beta_{n} \cap\,\mathcal{C})\,eyx$. That $\beta_{n} \cap\,\mathcal{C}$ is an idempotent pure congruence gives
 $eyx \in E_{\scriptscriptstyle{S}}$. We deduce that $\ker{\alpha_{n-1}}$ is $E$-reflexive.
\end{proof}

For a given congruence $\rho$, in a similar way to Lemma \ref{referee}'s we may find that the least $\ker{\alpha_n}$-is-Clifford congruence
containing $\rho$ exists. Denote it by $(\alpha_{n+2})_\rho$. We are now ready for characterizations of $\beta_n$-is-over $E$-unitary
congruences.

\begin{proposition}\label{boeuc}
 For $n \geqslant 1$, the following statements concerning a congruence $\rho$ on an inverse semigroup $S$ are equivalent.\\
 (1) $\rho$ is a $\beta_n$-is-over-$E$-unitary congruence;\\
 (2) $(\alpha_{n+1})_{\rho} \subseteq \rho^K$, where $(\alpha_{n+1})_{\rho}$ is the least $\ker{\alpha_{n-1}}$-is-Clifford congruence on $S$
 containing $\rho$;\\
 (3) $\ker{(\alpha_{n+1})_{\rho}}=\ker{\rho}$.
\end{proposition}
\begin{proof}
 $(1) \Rightarrow (2)$. The correspondence of congruences on $S$ containing $\rho$ and congruences on $S/\rho$ shows that for any $a, b \in S$,
 $$a\,(\alpha_{n+1})_\rho\,b \iff (a\rho)\,\alpha_{n+1}(S/\rho)\,(b\rho).$$ If $S/\rho$ is a $\beta_n$-is-over-$E$-unitary semigroup, then
 $(\alpha_{n+1})_{\rho}/\rho=\alpha_{n+1}(S/\rho) \subseteq \tau(S/\rho)=\rho^K/\rho$, and thus $(\alpha_{n+1})_{\rho} \subseteq \rho^K$.

 $(2) \Rightarrow (3)$. Since $\rho \subseteq (\alpha_{n+1})_{\rho} \subseteq \rho^K$, we have $\ker{\rho} \subseteq
 \ker{(\alpha_{n+1})_{\rho}} \subseteq \ker{\rho^K}=\ker{\rho}$, which implies $\ker{(\alpha_{n+1})_{\rho}}=\ker{\rho}$.

 $(3) \Rightarrow (1)$. $\ker{(\alpha_{n+1})_{\rho}}=\ker{\rho}$ implies $(\alpha_{n+1})_{\rho} \subseteq \rho^K$. Since
 $\ker{\rho}=\ker{(\alpha_{n+1})_{\rho}}$, $(\alpha_{n+1})_{\rho}/\rho$ is an idempotent pure congruence on $S/\rho$, which gives that $S/\rho$
 is a $\beta_n$-is-over-$E$-unitary semigroup by Proposition \ref{boeu}. This completes the proof that $\rho$ is a
 $\beta_n$-is-over-$E$-unitary congruence.
\end{proof}

We now turn to the set of all $\beta_n$-is-over-$E$-unitary congruences on an inverse semigroup. Recall that $\mathcal{B}_n$ is the set of all
congruences $\theta$ on $S$ such that $\beta_n(S/\theta)$ is over $E$-unitary semigroups, or equivalently, the set of all
$\beta_n$-is-over-$E$-unitary congruences on $S$ ordered by inclusion.

\begin{theorem}\label{boeuex}
 Let $S$ be an inverse semigroup.\\
 (1) $\mathcal{B}_n$ is a complete $\cap$-subsemilattice of $\mathcal{C}(S)$ with least element $\beta_{n+2}=(\alpha_{n+1})_k=(\alpha_{n+1}
 \cap \mathcal{L})^*$ and greatest element $\omega$;\\
 (2) the interval $[\beta_{n+2}, \alpha_{n+1}]$ is a complete sublattice of $\mathcal{B}_n$.
\end{theorem}
\begin{proof}
 (1) It follows immediately from Lemma \ref{referee} that $\mathcal{B}_n$ is a complete $\cap$-subsemilattice of $\mathcal{C}(S)$.

 To prove that $(\alpha_{n+1})_k$ is the least $\beta_n$-is-over-$E$-unitary congruence on $S$, we first note that
 $(\alpha_{n+1})_{(\alpha_{n+1})_k}=\alpha_{n+1} \subseteq \alpha_{n+1}^K=((\alpha_{n+1})_k)^K$ so that $(\alpha_{n+1})_k$ is a
 $\beta_n$-is-over-$E$-unitary congruence. If $\rho$ is a $\beta_n$-is-over-$E$-unitary congruence, then $\alpha_{n+1} \subseteq
 (\alpha_{n+1})_\rho \subseteq \rho^K$ and $(\alpha_{n+1})_k \subseteq (\rho^K)_k=\rho_k \subseteq \rho$, which implies that
 $\beta_{n+2}=(\alpha_{n+1})_k$ is the least $\beta_n$-is-over-$E$-unitary congruence.

 (2) The argument here goes along the same lines as in Theorem \ref{kccex}.
\end{proof}

We conclude this section with a new observation comparing to Petrich -- Reilly \cite[Theorem 5.5]{network}.

\begin{definition}
 An inverse semigroup $S$ might satisfy one of the following implications:\\
 (A$_0'$) $x=y$; (A$_1'$) $x^{-1}x=y^{-1}y$; (A$_2'$) $y \in E\zeta$;\\
 (A$_n'$) $xy=x$, $x^{-1}x\,\alpha_{n-2}\,yy^{-1} \Rightarrow y \in E\zeta$, $n \geqslant 3$;\\
 (B$_0'$) $x=y$; (B$_1'$) $y \in E$;\\
 (B$_n'$) $xy=x$, $x^{-1}x\,\alpha_{n-1}\,yy^{-1} \Rightarrow y \in E$, $n \geqslant 2$.
\end{definition}

We now come to the main theorem.

\begin{theorem}\label{net}
 For an inverse semigroup $S$,\\
 (1) $\alpha_n$ is the least congruence $\rho$ on $S$ such that $S/\rho$ satisfies (A$_n'$);\\
 (2) $\beta_n$ is the least congruence $\rho$ on $S$ such that $S/\rho$ satisfies (B$_n'$).
\end{theorem}
\begin{proof}
 We will first observe that the theorem is true for $n=0$, 1 and 2, and then complete the proof with an induction argument.

 The assertion of the theorem for $\alpha_0$, $\alpha_1$, $\alpha_2$, $\beta_0$ and $\beta_1$ follows directly from \cite[Theorem
 5.5]{network}. $\beta_2$, as we know, is the least $E$-unitary congruence, or the least $\beta_0$-is-over-$E$-unitary congruence. It follows
 from \cite[Proposition \@Roman3.7.2]{inverse} that $\beta_2$ is the least congruence $\rho$ such that $S/\rho$ satisfies (B$_2'$).

 Now suppose that $n \geqslant 3$ and that the theorem is valid for smaller integers. Then, by the induction hypothesis that $\beta_{n-1}$ is a
 $\beta_{n-3}$-is-over-$E$-unitary congruence, applying Proposition \ref{kercliffcon}, we obtain that $S/\alpha_n$ is a
 $\ker{\alpha_{n-2}}$-is-Clifford semigroup, which satisfies (A$_n'$) by virtue of Proposition \ref{kercliff}. Similarly, applying Proposition
 \ref{boeuc}, by the induction hypothesis that $\alpha_{n-1}$ is a $\ker{\alpha_{n-3}}$-is-Clifford congruence, we get that $S/\beta_n$ is a
 $\beta_{n-2}$-is-over-$E$-unitary semigroup, which satisfies (B$_n'$) in view of Proposition \ref{boeu}.

 The minimality of these congruences follows immediately from Theorem \ref{kccex} and Theorem \ref{boeuex}.
\end{proof}

\begin{remark}
 (1) We obtain by Theorem \ref{kccex} that $\eta_t$ is the least Clifford congruence, and that $\pi_t$ is the least $E\omega$-Clifford
 congruence, which is due to Wang - Feng \cite{feng}.

 (2) By Theorem \ref{boeuex} we get that $\sigma_k$ is the least $E$-unitary congruence, and that $(\pi_t)_k$ is the least
 $\pi$-is-over-$E$-unitary congruence. Proposition \ref{ker} shows that $S/(\pi_t)_k$ is an $E\omega$-$E$-reflexive semigroup. Here a
 correction should be made to Theorem 3.2 of \cite{feng}: $\pi$-is-over-$E$-unitary semigroups are $E\omega$-$E$-reflexive, but
 $E\omega$-$E$-reflexive semigroups are not necessarily $\pi$-is-over-$E$-unitary semigroups.
\end{remark}

The min network is redepicted in Figure \ref{new} together with the types of semigroups to which the quotient semigroups belong.

\begin{figure}[!hbt]\label{new}
 \renewcommand*\figurename{Figure}
 \renewcommand*\captionlabeldelim{}
 \setlength{\unitlength}{0.7cm}
 \begin{center}
  \begin{picture}(4,24)
   \drawline(1,3)(4,6)(0,10)(4,14)(0,18)(4,22)(2,24)(0,22)(4,18)(0,14)(4,10)(0,6)(4,2)(3,1)
   \allinethickness{0.7mm}
   \put(0,6){\circle*{0.06}} \put(0,10){\circle*{0.06}} \put(0,14){\circle*{0.06}} \put(0,18){\circle*{0.06}} \put(0,22){\circle*{0.06}}
   \put(2,24){\circle*{0.06}} \put(2,20){\circle*{0.06}} \put(2,16){\circle*{0.06}} \put(2,12){\circle*{0.06}} \put(2,8){\circle*{0.06}}
   \put(2,4){\circle*{0.06}} \put(4,2){\circle*{0.06}} \put(4,6){\circle*{0.06}} \put(4,10){\circle*{0.06}} \put(4,14){\circle*{0.06}}
   \put(4,18){\circle*{0.06}} \put(4,22){\circle*{0.06}}
   \put(2,24.3){\makebox(0,0)[b]{$\omega$}}
   \put(-0.3,22){\makebox(0,0)[r]{$\sigma=\alpha_1$}}
   \put(-0.3,18){\makebox(0,0)[r]{$\eta_t=\nu=\alpha_2$}}
   \put(-0.3,14){\makebox(0,0)[r]{$\pi_t=\alpha_3$}}
   \put(-0.3,10){\makebox(0,0)[r]{$\lambda_t=\alpha_4$}}
   \put(-0.3,6){\makebox(0,0)[r]{$(\beta_4)_t=\alpha_5$}}
   \put(4.3,22){\makebox(0,0)[l]{$\beta_1=\eta$}}
   \put(4.3,18){\makebox(0,0)[l]{$\beta_2=\pi=\sigma_k$}}
   \put(4.3,14){\makebox(0,0)[l]{$\beta_3=\lambda=\nu_k$}}
   \put(4.3,10){\makebox(0,0)[l]{$\beta_4=(\pi_t)_k$}}
   \put(4.3,6){\makebox(0,0)[l]{$\beta_5=(\lambda_t)_k$}}
   \put(4.3,2){\makebox(0,0)[l]{$\beta_6=(\alpha_5)_k$}}
   \put(-0.3,21.2){\makebox(0,0)[r]{\footnotesize{\textsf{group}}}}
   \put(-0.3,17.2){\makebox(0,0)[r]{\footnotesize{\textsf{Clifford}}}}
   \put(-0.3,13.2){\makebox(0,0)[r]{\footnotesize{$\ker{\sigma}$-\textsf{is-Clifford}}}}
   \put(-0.3,9.2){\makebox(0,0)[r]{\footnotesize{$\ker{\nu}$-\textsf{is-Clifford}}}}
   \put(-0.3,5.2){\makebox(0,0)[r]{\footnotesize{$\ker{\alpha_3}$-\textsf{is-Clifford}}}}
   \put(4.3,21.2){\makebox(0,0)[l]{\footnotesize{\textsf{semilattice}}}}
   \put(4.3,17.2){\makebox(0,0)[l]{\footnotesize{$E$-\textsf{unitary}}}}
   \put(4.3,13.2){\makebox(0,0)[l]{\footnotesize{$E$-\textsf{reflexive}}}}
   \put(4.3,9.2){\makebox(0,0)[l]{\footnotesize{$\pi$-\textsf{is-over-$E$-unitary}}}}
   \put(4.3,5.2){\makebox(0,0)[l]{\footnotesize{$\lambda$-\textsf{is-over-$E$-unitary}}}}
   \put(4.3,1.2){\makebox(0,0)[l]{\footnotesize{$\beta_4$-\textsf{is-over-$E$-unitary}}}}
   \put(2,0.5){\makebox(0,0)[c]{$\vdots$}}
  \end{picture}
  \caption{\quad min network of inverse semigroups} \label{origin1}
 \end{center}
\end{figure}
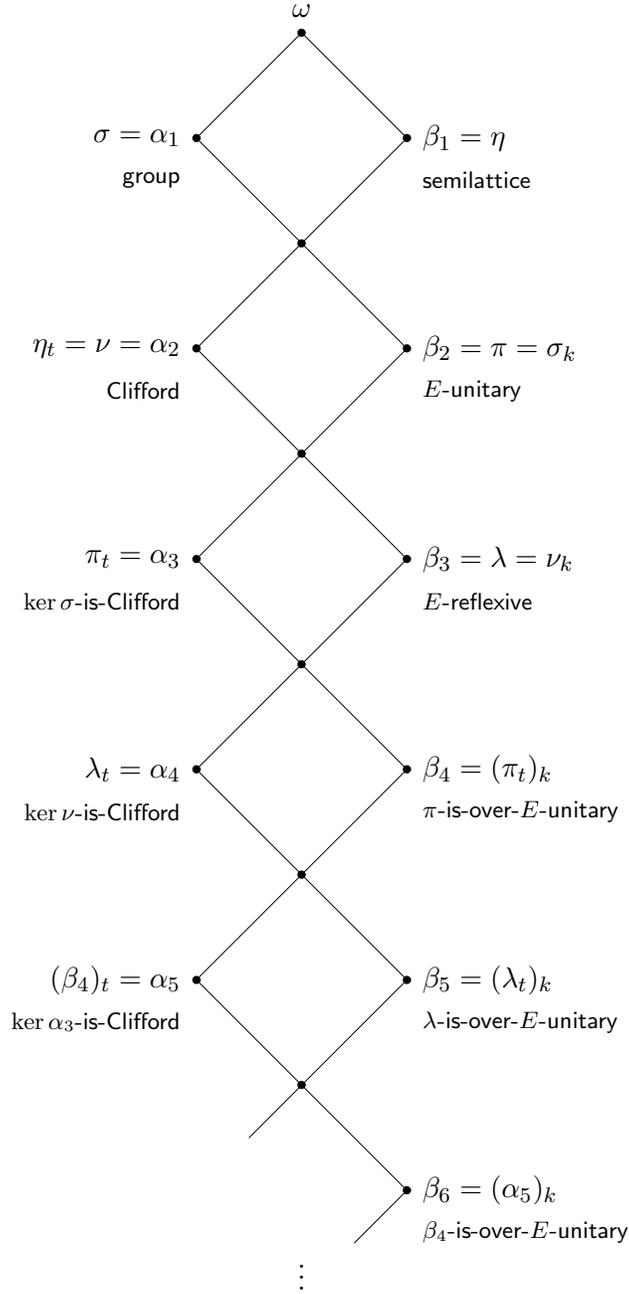

Refer to any inverse semigroup satisfying (A$_n$) as an \emph{$A_n$-semigroup}. Similarly, an inverse semigroup satisfying (B$_n$) is called a
\emph{$B_n$-semigroup}. The next proposition gives one further observation concerning the min network.

\begin{proposition}\label{quotient}
 Let $m$, $n$ be nonnegative integers.\\
 (1) $\alpha_{m+n}$ is the least congruence $\rho$ on $S$ such that $\alpha_n(S/\rho)$ is over $A_m$-semigroups;\\
 (2) $\beta_{m+n}$ is the least congruence $\rho$ on $S$ such that $\beta_n(S/\rho)$ is over $B_m$-semigroups.
\end{proposition}
\begin{proof}
 We shall only prove that $(e\alpha_{m+n})(\alpha_n(S/\alpha_{m+n}))=(e\alpha_{m+n})(\alpha_n/\alpha_{m+n})$ satisfies (A$_m$) for any $e \in
 E_S$. Denote $(e\alpha_{m+n})(\alpha_n/\alpha_{m+n})=\{a\alpha_{m+n}\,|\,a\,\alpha_n\,e\}$ by $E_0$. For $m \geqslant 3$, notice that
 $\beta_{m-3}(E_0)=(\beta_{n+m-3}/\alpha_{m+n})|_{E_0}$. Suppose that $a\alpha_{m+n}$, $b\alpha_{m+n} \in E_0$ with
 $(a\alpha_{m+n})(b\alpha_{m+n})=a\alpha_{m+n}$ and $(a\alpha_{m+n}) \beta_{m-3}(E_0) (b\alpha_{m+n})$. Then
 $(a\alpha_{m+n})(b\alpha_{m+n})=a\alpha_{m+n}$ and $a\,\beta_{m+n-3}\,b$. But $S/\alpha_{m+n}$ satisfies (A$_{m+n}$) and $b\alpha_{m+n} \in
 E_{S/\alpha_{m+n}}\zeta$ and hence $b\alpha_{m+n} \in E_{E_0}\zeta$. Thus $E_0$ satisfies (A$_m$). The minimality of the congruence follows
 immediately from the fact that $\alpha_{m+n}$ is the least congruence $\gamma$ on $S$ such that $S/\gamma$ satisfies (A$_m'$).

 The remaining arguments go along the same lines and are omitted.
\end{proof}

\section{Coincidences}

Petrich \cite{inverse} investigates necessary and sufficient conditions in order that two of the congruences in $\{\omega, \sigma, \eta, \nu,
\pi, \lambda, \mu, \tau, \varepsilon\}$ coincide. This creates many interesting classes of inverse semigroups. Further equivalent conditions
can be established if $\alpha_n$ and $\beta_n$ are taken into account.

\begin{proposition}
 The following statements hold in any inverse semigroups.\\
 (1) For $n \geqslant 2$, $\alpha_n=\omega \iff \sigma=\eta=\omega \iff \beta_n=\omega$;\\
 (2) for $n \geqslant 3$, $\alpha_n=\sigma \iff \beta_{n-1}=\sigma$;\\
 (3) for $n \geqslant 2$, $\alpha_n=\eta \iff \beta_{n+1}=\eta$;\\
 (4) for $n \geqslant 4$, $\alpha_n=\nu \iff \beta_{n-1}=\nu$;\\
 (5) for $n \geqslant 3$, $\alpha_n=\pi \iff \beta_{n+1}=\pi$;\\
 (6) for $n \geqslant 4$, $\alpha_n=\lambda \iff \beta_{n+1}=\lambda$;\\
 (7) for $n \geqslant 3$, $\alpha_n=\mu \iff S$ is a $\beta_{n-3}$-is-over-$E$-unitary fundamental inverse semigroup;\\
 (8) for $n \geqslant 1$, $\alpha_n=\tau \iff S$ is a $\beta_{n-1}$-is-over-$E$-unitary semigroup with $\tr{\tau}=\tr{\beta_{n-1}}$;\\
 (9) for $n \geqslant 2$, $\beta_n=\tau \iff S$ is a $\beta_{n-2}$-is-over-$E$-unitary $E$-disjunctive inverse semigroup.
\end{proposition}
\begin{proof}
 (1) Suppose that $\alpha_n=\omega$. Since $\alpha_n \subseteq \sigma$ and $\alpha_n \subseteq \eta$, it follows that $\sigma=\eta=\omega$.
 Conversely, if $\sigma=\eta=\omega$, then $\pi=\sigma_k=\omega_k=\eta=\omega$ and $\nu=\eta_t=\omega_t=\sigma=\omega$. Similarly, we have
 $\lambda=\omega$ and $\pi_t=\omega$. Inductively, we have $\alpha_n=\omega$. $\sigma=\eta=\omega \iff \beta_n=\omega$ follows by duality.

 (2) If $\alpha_n=\sigma$, then $\sigma=\alpha_n \subseteq \beta_{n-1} \subseteq \beta_2 \subseteq \sigma$ and thus $\beta_{n-1}=\sigma$.
 Conversely, if $\beta_{n-1}=\sigma$, then $\alpha_n=(\beta_{n-1})_t=\sigma_t=\sigma$.

 (3) If $\alpha_n=\eta$, then $\beta_{n+1}=(\alpha_n)_k=\eta_k=\eta$. Conversely, if $\beta_{n+1}=\eta$, then $\beta_{n+1}=(\alpha_n)_k
 \subseteq \alpha_n \subseteq \eta$ and thus $\alpha_n=\eta$.

 (4) If $\alpha_n=\nu$, then $\nu=\alpha_n=(\beta_{n-1})_t \subseteq \beta_{n-1} \subseteq \beta_3 \subseteq \nu$ and thus $\beta_{n-1}=\nu$.
 Conversely, if $\beta_{n-1}=\nu$, then $\alpha_n=(\beta_{n-1})_t=\nu_t=\nu$.

 (5) If $\alpha_n=\pi$, then $\beta_{n+1}=(\alpha_n)_k=\pi_k=\pi$. Conversely, if $\beta_{n+1}=\pi$, then $\pi=\beta_{n+1}=(\alpha_n)_k
 \subseteq \alpha_n \subseteq \alpha_3 \subseteq \pi$ and thus $\alpha_n=\pi$.

 (6) If $\alpha_n=\lambda$, then $\beta_{n+1}=(\alpha_n)_k=\lambda_k=\lambda$. Conversely, if $\beta_{n+1}=\lambda$, then $\lambda=\beta_{n+1}
 \subseteq \alpha_n \subseteq \alpha_4 \subseteq \lambda$ and thus $\alpha_n=\lambda$.

 (7) For $n=3$, the assertion follows directly from \cite[Proposition 4.3]{feng}. For $n>3$, suppose that $\alpha_n=\mu$. Since
 $\alpha_n=(\beta_{n-1})_t \subseteq \beta_{n-1} \subseteq \alpha_n^T=\mu^T=\mu=\alpha_n$, it follows that $\alpha_n=\beta_{n-1}=\mu$ and thus
 $\mu=\alpha_n=(\beta_{n-1})_t=\mu_t=\varepsilon$, which implies $\beta_{n-1}=\varepsilon$. Thus $\mu=\varepsilon$ gives that $S$ is
 fundamental while $\beta_{n-1}=\varepsilon$ gives that $S$ is a $\beta_{n-3}$-is-over-$E$-unitary semigroup.

 If $S$ is a $\beta_{n-3}$-is-over-$E$-unitary fundamental inverse semigroup, then $\mu=\varepsilon$ and $\beta_{n-1}=\varepsilon$, which imply
 that $\beta_{n-1}=\mu$. Hence $\alpha_n=(\beta_{n-1})_t=\mu_t=\varepsilon=\mu$.

 (8) For $n=1$ and $n=2$, the assertions follow directly from \cite[Coincidences \@Roman3.8.10]{inverse}. For $n \geqslant 3$, if
 $\alpha_n=\tau$, then $\tr{\tau}=\tr{\alpha_n}=\tr{\beta_{n-1}}$ and $\beta_{n+1}=(\alpha_n)_k=\tau_k=\varepsilon$, which give that $S$ is a
 $\beta_{n-1}$-is-over-$E$-unitary semigroup.

 Now suppose that $S$ is a $\beta_{n-1}$-is-over-$E$-unitary semigroup with $\tr{\tau}=\tr{\beta_{n-1}}$. The hypothesis implies that
 $\beta_{n+1}=\varepsilon$ so that $\ker{\alpha_n}=\ker{\beta_{n+1}}=E_{\scriptscriptstyle{S}}=\ker{\tau}$, and thus
 $\tr{\alpha_n}=\tr{\beta_{n-1}}=\tr{\tau}$ gives that $\alpha_n=\tau$.

 (9) For $n=2$, the assertion follows directly from \cite[Coincidences \@Roman3.8.10]{inverse}. For $n>2$, if $\beta_n=\tau$, then
 $\tau=\beta_n=(\beta_n)_k=\tau_k=\varepsilon$, which gives that $S$ is a $\beta_{n-2}$-is-over-$E$-unitary semigroup and is $E$-disjunctive.
 Conversely, if $S$ is a $\beta_{n-2}$-is-over-$E$-unitary $E$-disjunctive inverse semigroup, then $\beta_n=\varepsilon$ and $\tau=\varepsilon$
 which imply that $\beta_n=\tau=\varepsilon$.
\end{proof}

\textbf{Acknowledgements} The authors are grateful to the careful referee for thoughtful comments and insights which helped to improve the paper, in particular, with regard to Lemma \ref{referee} and Proposition \ref{quotient}. The first author would like to thank Professor
Victoria Gould for her continuing support and encouragement. This work is supported by a Grant of the National Natural Science Foundation of
China (11871150) and a Grant of the Ministry of Education of China (18YJCZH206).

\end{document}